\documentclass[12pt]{amsart}
\usepackage{enumerate, amscd, amssymb, fullpage}

\newtheorem{theorem}{Theorem}[section]
\newtheorem{lemma}[theorem]{Lemma}

\newtheorem{corollary}[theorem]{Corollary}
\theoremstyle{definition}

\def\a{\alpha}
\def\s{\sigma}

\begin{document}




\title{Symmetry in the Geometry of Metric Contact Pairs}



\author{G. Bande}
\address{Universit\`a degli Studi di Cagliari\\
Dipartimento di Matematica e Informatica\\
Via Ospedale 72, 09124 Cagliari, ITALIA}
\email{gbande{\char'100}unica.it}

\author{D. E. Blair}
\address{Department of Mathematics\\
Michigan State University\\
 East Lansing, MI 48824--1027, USA}
 \email{blair{\char'100}math.msu.edu}

\thanks{The second author was supported by a Visiting Professor fellowship at the Universit\`a degli Studi di Cagliari in April 2011, financed by Regione Autonoma della Sardegna.}

\begin{abstract}
We prove that the universal covering of a complete locally symmetric normal metric contact pair manifold is a Calabi-Eckmann manifold. Moreover we show that a complete, simply connected, normal metric contact pair manifold such that the foliation induced by the vertical subbundle is regular and reflections in the integral submanifolds of the vertical subbundle are isometries, then the manifold is the product of  globally $\phi$-symmetric spaces and fibers over a locally symmetric space endowed with a symplectic pair.
\end{abstract}

\maketitle                   






\section{Introduction}

 In real contact geometry the question of locally symmetric contact metric manifolds has a long history and a short answer.  Already in 1962 Okumura \cite{O} proved that a locally symmetric Sasakian manifold is locally isometric to the sphere $S^{2n+1}(1)$  and in 2006 Boeckx and Cho \cite{BC} proved that a locally symmetric contact metric manifold is locally isometric to $S^{2n+1}(1)$ or to 
 $E^{n+1}\times S^n(4)$. Various studies and generalizations of this question were made in the intervening years.  Most importantly, since the locally symmetric condition is clearly very restrictive, Takahashi \cite{Ta} introduced the notion of a locally $\phi$-symmetric space for Sasakian manifolds by restricting the locally symmetric condition to the contact subbundle and showed that these manifolds locally fiber over Hermitian symmetric spaces. The second author and Vanhecke \cite{BV} showed that this condition is equivalent to reflections in the integral curves of the Reeb vector field being isometries. For a general discussion of these ideas in real contact geometry we refer the reader to \cite{B}.
 
 In this paper we begin the study of these ideas for metric contact pairs (or bicontact manifolds). 
 In \cite{BH1} A. Hadjar and the first author introduced the notion of normality for a contact pair and we first show that for a locally symmetric normal metric contact pair, the universal covering space of such a manifold is a Calabi-Eckmann manifold, i.e. $S^{2m+1}(1)\times S^{2n+1}(1)$,

  We then study reflections in the integral submanifolds of the vertical subbundle of a normal metric contact pair. Suppose the induced foliation by the integral submanifolds of the vertical subbundle is regular giving us a fibration.  When such reflections are isometries we show that the manifold is the product of  locally $\phi$-symmetric spaces and fibers over a locally symmetric space with a symplectic pair structure.

\section{Preliminaries}

Contact pairs were introduced by G. D. Ludden, K. Yano and the second author in \cite{BLY} under the name
{\it bicontact} and by A. Hadjar and the first author in \cite{Bande1, BH} with the name {\it contact pair}.
A pair of 1-forms $(\a_1,\a_2)$ on a manifold $M$ is said to be a {\it contact pair of type $(m,n)$} if
\vskip4pt
\centerline{$\a_1\wedge(d\a_1)^m\wedge\a_2\wedge(d\a_2)^n$ is a volume form,}
\vskip4pt
\centerline{$(d\a_1)^{m+1}=0$ and $(d\a_2)^{n+1}=0$.}
\vskip4pt
\noindent
While it is possible to consider a contact pair of type (0,0) (see \cite{BH}), it seems most natural to require at least one of the forms to resemble a contact form.  Thus we adopt the convention that $m \geq 1$ and
$n \geq 0$.

We can naturally associate to a contact pair two subbundles 
$$\{ X: \a_i(X)=0, d\a_i(X,Y)=0\; \forall Y\},.\; i=1,2$$
These subbundles are integrable \cite {BH} and determine the {\it characteristic foliations} of $M$, denoted 
${\mathcal F}_1$ and ${\mathcal F}_2$ respectively.

The equations
\begin{eqnarray*}
&\alpha_1 (Z_1)=\alpha_2 (Z_2)=1  , \; \; \alpha_1 (Z_2)=\alpha_2
(Z_1)=0 \, , \\
&i_{Z_1} d\alpha_1 =i_{Z_1} d\alpha_2 =i_{Z_2}d\alpha_1=i_{Z_2}
d\alpha_2=0 \, ,
\end{eqnarray*}
where $i_X$ is the contraction with the vector field $X$, determine uniquely the two vector fields $Z_1$ and $Z_2$, called \emph{Reeb vector fields}. Since they commute, they give rise to a locally free $\mathbb{R}^2$-action, called  the \emph{Reeb action}.

A {\it contact pair structure} \cite{BH3} on a manifold $M$ is a triple
$(\alpha_1 , \alpha_2 , \phi)$, where $(\alpha_1 , \alpha_2)$ is a
contact pair and $\phi$ a tensor field of type $(1,1)$ such that:
$$\phi^2=-Id + \alpha_1 \otimes Z_1 + \alpha_2 \otimes Z_2 , \quad
\phi Z_1=\phi Z_2=0$$
where $Z_1$ and $Z_2$ are the Reeb vector fields of $(\alpha_1 ,
\alpha_2)$.

One can see that $\alpha_i \circ \phi =0$ for $i=1,2$ and that the rank of $\phi$ is
equal to $\dim M -2$.  Since we are also interested in the induced structures, we recall the notion of the decomposability of $\phi$ which we will assume throughout this paper.
The endomorphism $\phi$ is said to be {\it decomposable} \cite{BH3} if
$\phi (T\mathcal{F}_i) \subset T\mathcal{F}_i$, for $i=1,2$.

If $\phi$ is decomposable, then $(\alpha_1 , Z_1 ,\phi)$ (respectively
$(\alpha_2 , Z_2 ,\phi)$) induces, on every leaf of $\mathcal{F}_2$ (respectively $\mathcal{F}_1$), a contact form and the restriction  $\phi_i$ of $\phi$ to the leaf forms an almost contact structure 
$(\alpha_i, Z_i,\phi_i)$.

In \cite{BH1} A. Hadjar and the first author introduced the notion of {\it normality} for a contact pair structure $(\a_1,\a_2,\phi)$ as the integrability of two natural almost complex structures on $M$.  This is equivalent to the equation
$$N^1(X,Y)=:[\phi,\phi](X,Y)+2d\a_1(X,Y)Z_1+2d\a_2(X,Y)Z_2=0,$$
$[\phi,\phi]$ being the Nijenhuis tensor of $\phi$.
We also note the following tensors
$$N^2 _i (X,Y) =: (\pounds_{\phi X} \alpha_i) (Y) - (\pounds_{\phi Y}\alpha_i)(X) ,  \; i=1,2.$$

On manifolds endowed with contact pair structures it is natural
to consider the following metrics \cite{BH3}.
Let $(\alpha_1 , \alpha_2 ,\phi )$ be a contact pair structure on
a manifold $M$, with Reeb vector fields $Z_1$ and $Z_2$. A
Riemannian metric $g$ on $M$ is said to be \emph{associated} if 
$g(X, \phi Y)= (d \alpha_1 + d\alpha_2) (X,Y)$ and $g(X, Z_i)=\alpha_i(X)$, for $i=1,2$ and for
all $X,Y \in \Gamma (TM)$.

A {\it metric contact pair} on a manifold $M$ is a
four-tuple $(\alpha_1, \alpha_2, \phi, g)$ where $(\alpha_1,
\alpha_2, \phi)$ is a contact pair structure and $g$ an associated
metric with respect to it. The manifold $M$ with this structure will be called  a {\it metric contact pair}.

Observe that for a metric contact pair,  $(\alpha_1, \alpha_2, \phi, g)$, the endomorphism field $\phi$ is decomposable
if and only if the characteristic foliations $\mathcal{F}_1 ,
\mathcal{F}_2$ are orthogonal \cite{BH2}. In this case $(\alpha_i, \phi, g)$ induces a contact metric structure $(\alpha_i, \phi_i, g)$ on the leaves of $\mathcal{F}_j$ , for $j\neq i$ .
Moreover by the normality each $(\alpha_i, \phi_i, g)$ is a Sasakian structure on each leaf.

For a normal metric contact pair, $(\alpha_1, \alpha_2, \phi , g)$, with decomposable $\phi$ we have
(see \cite{BH1})
$$N^2 _1=N^2 _2=0,\quad \pounds_{Z_1}\phi=\pounds_{Z_2}\phi=0.$$
Moreover the vector fields $Z_1$ and $Z_2$ are Killing \cite{BH2}.

Also the vector field $Z=Z_1+Z_2$ plays an important role.   In particular we have the following basic formulas \cite{BH2}.
$$2 g((\nabla_X \phi)Y, W)= g \left (N^1 (Y,W), \phi X \right )$$
$$+ 2 \sum_{i=1} ^2 \bigl (d\alpha_i (\phi Y , X) \alpha_i (W)
- d\alpha_i (\phi W,X) \alpha_i (Y)\bigr ),\eqno(2.1)$$
$$\nabla_XZ=-\phi X,\quad R_{XZ}Z=-\phi^2X.$$

\begin{lemma} On a normal metric contact pair, for horizontal vector fields $X$ and $W$
we have
$$R_{W\,Z}X=-d\a_1(\phi W,X)Z_1-d\a_2(\phi W,X)Z_2.$$
In particular $R_{\phi X\,Z}X=0$.
\end{lemma}
\begin{proof} Since $\nabla_YZ=-\phi Y$ we have
$$R_{X\,Y}Z=-\nabla_X\phi Y+\nabla_Y\phi X-\phi[X,Y]
=-(\nabla_X\phi)Y+(\nabla_Y\phi)X.$$
Taking the inner product with a vector $W$ and using equation (2.1) for the covariant derivative of $\phi$ one
readily obtains
$$g(R_{X\,Y}Z,W)=d\a_1(\phi W,X)\a_1(Y)+d\a_2(\phi W,X)\a_2(Y)$$
$$-d\a_1(\phi W,Y)\a_1(X)-d\a_2(\phi W,Y)\a_2(X)$$
$$=-g(R_{W\,Z}X,Y).$$
Now taking $X$ and $W$ horizontal and $Y$ arbitrary, the result follows.
\end{proof}

We close this section with a brief description of the equations of submanifold theory.  For a submanifold of a Riemannian manifold $(M,g)$ we will use the same letter for the induced metric but denote the induced connection by $\dot\nabla$.  The second fundamental form will be denoted by $\s$ and the Weingarten map for a given unit normal $W$ will be denoted by $A_W$.  We denote by $\nabla^\perp$ the connection in the normal bundle. Let $X,Y,U,V$ be tangent vectors and $W_i, W_2$ unit normals. Define the covariant derivative of $\s$, $\nabla'\s$ by
$$(\nabla'\s)(X,Y,U)=\nabla^\perp\s(Y,U)-\s(\nabla_XY,U)+\s(Y,\nabla_XU).$$
The equations of Gauss, Codazzi and Ricci-K\"uhne are then respectively
$$R(X,Y,U,V)=\dot R(X,Y,U,V)+g(\s(X,U),\s(Y,V))-g(\s(Y,U),\s(X,V)),$$
$$(R_{XY}U)^\perp=(\nabla'\s)(X,Y,U)-(\nabla'\s)(Y,X,U),$$
$$R(X,Y,W_1,W_2)=R^\perp(X,Y,W_1,W_2)-g([A_{W_1},A_{W_2}]X,Y).$$

\section{Locally Symmetric Normal Metric Contact Pairs}

In this section we prove that a complete, locally symmetric, normal, metric contact pair is either compact and
its universal covering space is a Calabi-Eckmann manifold, $S^{2m+1}(1)\times S^{2n+1}(1)$, or  its universal covering space is $S^{2m+1}(1)\times{\mathbb R}$.
There are several aspects to the proof, namely that the characteristic foliations are totally geodesic, 
that at least one of the factor spaces has constant curvature +1, and that we can lift the resulting local Riemannian product structure to the universal covering space..  One might approach this problem at some point in the proof by noting that local symmetry implies that the eigenspaces of the Ricci tensor are integrable with totally geodesic leaves and that the leaves are already known to be Sasakian manifold which are irreducible (\cite{Tanno}). If after computing the curvature, one would then, more or less, be done except for the case that the manifold is Einstein which occurs for a Calabi-Eckmann manifold where the unit spheres have the same dimension.  We therefore prove our result directly.

\begin{theorem} Let $M$ be a complete, locally symmetric, normal, metric contact pair.  Then either the universal covering space of $M$ is a Calabi-Eckmann manifold, $S^{2m+1}(1)\times S^{2n+1}(1)$, and $M$ is compact, or 
the universal covering space of $M$ is $S^{2m+1}(1)\times{\mathbb R}$.
\end{theorem}
\begin{proof} We begin with the observation that since our manifold is locally symmetric, it is semi-symmetric, i.e. $R\cdot R=0$, so that
$$R(R_{X\,Y}X_1,X_2,X_3,X_4)+R(X_1,R_{X\,Y}X_2,X_3,X_4)$$
$$+R(X_1,X_2,R_{X\,Y}X_3,X_4)+R(X_1,X_2,X_3,R_{X\,Y}X_4)=0.$$
Taking $X_1=X_3=Y=Z$ and $X$ and $X_2$ horizontal, and recalling that $R_{X\,Z}Z=-\phi^2 X$, we have
$$0=R(X,X_2,Z,X_4)+R(Z,X_4,Z,R_{X\,Z}X_2)$$
$$+R(Z,X_2,X,X_4)+R(Z,X_2,Z,R_{X\,Z}X_4)$$
$$=R(X,X_2,Z,X_4)+g(\phi^2X_4,R_{X\,Z}X_2)+R(Z,X_2,X,X_4)-g(X_2,R_{X\,Z}X_4).$$
Expanding $\phi^2X_4$ this yields
$$2R(X,X_2,Z,X_4)-R(X,Z,X_2,X_4)$$
$$+\a_1(X_4)R(X,Z,X_2,Z_1)+\a_2(X_4)R(X,Z,X_2,Z_2)=0.$$
Taking $X_4=Z_1$, we have
$R(X,X_2,Z_1,Z_2)=0$ which together with the Lemma imply that the curvature for horizontal vectors $X,Y$ satisfies 
$$R_{X\,Y}Z=0.\eqno(3.1)$$

Now set $Y=Z$. Making the choice $X_1=X$, denoting $X_2$ by $Y$, $X_3$ by $U$ all horizontal, and setting $X_4=Z$, 
$R\cdot R=0$ gives
$$0=R(R_{X\,Z}X,Y,U,Z)+R(X,R_{X\,Z}Y,U,Z)$$
$$+R(X,Y,R_{X\,Z}U,Z)+R(X,Y,U,X).$$
The third term vanishes by (3.1). If now $X,Y,U$ are in $T{\mathcal F}_2$ the Lemma gives
$$R(X,Y,U,X)=g(X,X)g(U,Y)-g(U,X)g(X,Y).\eqno(3.2)$$
In particular if $Y=U\perp  X$, we have that the sectional  curvature
$$K(X,Y)=1.\eqno (3.3)$$

Now consider a leaf $L$ of the foliation ${\mathcal F}_2$ as a submanifold.  First of all $\nabla_{Z_1} Z_1=0$
gives $\dot\nabla_{Z_1} Z_1=0$ and $\s(Z_1,Z_1)=0$.
For $X$ tangent to the leaf we compute
$\nabla_XZ$ in two ways:
$$\nabla_XZ=-\phi X=-\phi_1X,$$
$$\nabla_XZ=\nabla_XZ_1+\nabla_XZ_2
=\dot\nabla_XZ_1+\s(X,Z_1)-A_{Z_2}X+\nabla_X^\perp Z_2.$$
Since the induced structure is Sasakian, we already know that $\dot\nabla_XZ_1=-\phi_1X$ and we then have that 
$$A_{Z_2}=0\;{\rm and}\;\s(X,Z_1)=-\nabla_X^\perp Z_2.\eqno (3.4)$$

For two vectors $X$ and $Y$ tangent to $L$ and a normal $W$ together with the normal $Z_2$ we have the equation of Ricci-K\"uhne
$$g(R_{X\,Y}Z_2,W)=g(R_{X\,Y}^\perp Z_2,W)+g([A_{Z_2},A_W]X,Y).$$
Using (3.4) this becomes
$$g(R_{X\,Y}Z_2,W)=g(-\nabla_X^\perp\s(Y,Z_1)+\nabla_Y^\perp\s(X,Z_1)+\s([X,Y],Z_1),W)$$
$$=g(-(\nabla'\s)(X,Y,Z_1)-\s(Y,-\phi_1X)+(\nabla'\s)(Y,X,Z_1)+\s(X,-\phi_1Y),W)$$
$$=-g(R_{X\,Y}Z_1,W)+g(\s(Y,\phi_1X)-\s(X,\phi_1Y),W)$$
by the Codazzi equation.
Combining the curvature terms and using the Lemma we readily have $g(R_{X\,Y}Z,W)=0$
and hence that 
$\s(X,\phi_1Y)=\s(\phi_1X,Y)$ and in turn that $\s(X,Z_1)=0$.
We know that $L$ with the induced structure is Sasakian and hence that $\dot K(X,Z_1)=1$;
using the Gauss equation and $\s(X,Z_1)=0$ we therefore have 
$$K(X,Z_1)=1.\eqno(3.5)$$

A plane section oblique to $Z_1$ is spanned by a horizontal unit vector $X$ and a unit vector of the form
$aZ_1+bY$, $a^2+b^2=1$ where $Y$ is orthogonal to both $Z_1$ and $X$.
From the Gauss equation we have
$$R(X,Y,Z_1,X)=\dot R(X,Y,Z_1,X)-\s(Y,Z_1)\s(X,X)+\s(X,Z_1)\s(X,Y),$$
but on a Sasakian manifold $\dot R_{X\,Y}Z_1$ vanishes for  $X$ and $Y$ horizontal and we have already seen that $\s$ vanishes on $Z_1$. 
Therefore $R(X,Y,Z_1,X)$ vanishes and we have that the sectional curvature of the oblique section is also +1. Thus, since sectional curvatures of a curvature tensor on a vector space determine that tensor, we have
that the curvature of the ambient metric restricted to a tangent space to $L$ satisfies
$$R(X,Y,U,V)=g(Y,U)g(X,V)-g(X,U)g(Y,V).\eqno(3.6)$$

Returning to the semi-symmetric condition, the choices $X,X_1,X_2$ horizontal in $T{\mathcal F}_2$, $X_3$
horizontal in $T{\mathcal F}_1$ and $X_4=Y=Z$, we have 
$$R(X_1,X_2,X_3,X)=0$$
 by the Lemma.  The Codazzi equation gives the same with $X_3=Z_2$.  Therefore for horizontal tangents $X,Y,U$ and any normal $W$
$$R(X, Y,U,W)=0.\eqno (3.7)$$
Similarly for  $X,Y$ horizontal tangents and $W_1, W_2$ horizontal normals,
 we have
$$R(X,W_1,W_2,Y)=0.\eqno(3.8)$$

We now utilize the condition of local symmetry.  Choose $X_4$ horizontal in $T{\mathcal F}_1$ and the other vector fields to be horizontal tangents, then we have
$$0=XR(X_1,X_2,X_3,X_4)-R(\nabla_XX_1,X_2,X_3,X_4)-(X_1,\nabla_XX_2,X_3,X_4)$$
$$-R(X_1,X_2,\nabla_XX_3,X_4)-R(X_1,X_2,X_3,\nabla_XX_4).$$
The first term vanishes by (3.7).  For the second term decompose $\nabla_XX_1$
as $\dot\nabla_XX_1+\s(X,X_1)$, then the second term vanishes by (3.7) and (3.8).  The third and fourth terms vanish in the same manner.  Using (3.7) again, the last term yields
$$R(X_1,X_2,X_3,-A_{X_4}X)=0.$$  Applying equation (3.6) with$X_2=X_3\perp X_1$ yields
$g(X_1,A_{X_4}X)=0$.
Since $X_4$ was an arbitrary normal and we saw earlier that $A_{Z_2}=0$, we have that $\s$
vanishes on horizontal tangent vectors. We have also seen that $\s(Z_1,X)=0$ for any $X$.  Therefore the leaves of
${\mathcal F}_2$ are totally geodesic submanifolds.  By the same argument the leaves of ${\mathcal F}_1$ are totally geodesic giving $M$ a local  Riemannian product structure.  

Furthermore the leaves of the foliations are Sasakian manifolds of constant curvature +1 or possibly the second factor space is 1-dimensional. 
Since $M$ is complete, the leaves are complete because they are totally geodesic.  In the first case, since they have constant curvature +1, their universal coverings are unit spheres. By a theorem of Blumenthal and Hebda \cite{BluHeb} the universal covering is then the product of two spheres and $M$ is compact.  Similarly in the second case the universal covering is $S^{2m+1}(1)\times{\mathbb R}$.
\end{proof}
As was shown in \cite{BK}, the existence of a normal MCP of type $(h,0)$ on a manifold is equivalent to saying that the manifold is a non-K\"ahler Vaisman manifold. Then we have:

\begin{corollary} 
Let $M$ be a complete locally symmetric non-K\"ahler Vaisman manifold. Then the Riemannian universal covering of $M$ is (up to constant scale of the metric) isometric to $S^{2m+1}(1)\times{\mathbb R}$.
\end{corollary}

\section{Reflections in the Vertical Foliation}

As we have seen the condition of local symmetry for a normal metric contact pair is extremely strong.  We therefore consider a weaker, but very geometric, condition in terms of local reflections in the integral submanifolds of the vertical subbundle. To do this we first recall the notion of a local reflection in a submanifold.
Given a Riemannian manifold $(M,g)$ and a submanifold $N$, {\it local reflection} in $N$, $\varphi_N$, is defined as follows.  For $m\in M$ sufficiently close to $N$ consider the minimal geodesic from $m$ to $N$ meeting $N$ orthogonally at $p$.  Let $X$ be the unit vector at $p$ tangent to the geodesic in the direction toward $m$.  Then
$\varphi_N$ maps $m=\exp_p(tX)\longrightarrow\exp_p(-tX)$.  In \cite{CV} Chen and Vanhecke gave the following necessary and sufficient conditions for a reflection to be isometric.

\begin{theorem}[Chen \& Vanhecke]  Let $(M,g)$ be a Riemannian manifold and $N$ a submanifold.  Then
the reflection  $\varphi_N$ is a local isometry if and only if
\begin{enumerate}
\item
$N$ {\rm  is totally geodesic;}
\item
$(\nabla^{2k}_{X\cdots X}R)(X,Y)X$  {\rm  is normal to N,}

$(\nabla^{2k+1}_{X\cdots X}R)(X,Y)X$  {\rm   is tangent to N and}

$(\nabla^{2k+1}_{X\cdots X}R)(X,V)X$  {\rm  is normal to N}
\end{enumerate}
for all vectors $X$, $Y$ normal to $N$ and vectors $V$ tangent to $N$ and all $k\in{\mathbb N}$.

\end{theorem}
Next recall that a foliation is {\it regular} if each point of the manifold has a neighborhood such that any leaf passing through the neighborhood passes through only once.         

We now prove the main result of this section.

\begin{theorem} 
Let $M$ be a complete, simply connected, normal metric contact pair and suppose that the foliation induced by vertical subbundle is regular.   If reflections in the integral submanifolds of the vertical subbundle are isometries, then the manifold is the product of  globally $\phi$-symmetric spaces
or, in the case of type $(m,0)$, the product of a globally $\phi$-symmetric space and $\mathbb R$.
Moreover $M$  fibers over a locally symmetric space with a symplectic pair structure.
\end{theorem}

\begin{proof}
First note that since $Z_1$ and $Z_2$ are Killing vector fields on a normal metric contact pair, the metric is projectable  to a Riemannian metric $g'$ on $M'$.  

Next we  observe that if $\gamma(s)$ is a geodesic on a normal metric contact pair which is initially orthogonal to $\mathcal V$, then it remains orthogonal to $\mathcal V$ for all $s$.  To see this we have only to note that
$$\gamma'g(\gamma',Z_1)=g(\gamma',-\phi_1\gamma')=0$$
and similarly for $Z_2$. Thus horizontal geodesics are projectable to geodesics on $M'$.

 Therefore the geodesic symmetries on $(M',g')$ are isometries and hence $(M',g')$ is a locally symmetric space.
 
 We now compare the covariant derivative of the curvature of $g'$  with that of $g$ acting on horizontal vectors.  For a general vector field $X$ on $M'$ we denote by $X^*$ its horizontal lift to $M$.  Two of the fundamental equations of a Riemannian submersion in the present context are the following (see
 \cite{ON} or \cite{FIP} for a general treatment)
 $$\nabla_{X^*}Y^*=(\nabla'_XY)^*+\a_1(\nabla_{X^*}Y^*)Z_1+\a_2(\nabla_{X^*}Y^*)Z_2,$$
 $$R(X^*,Y^*,U^*,V^*)=R'(X,Y,U,V)$$
$$+2\big(\a_1(\nabla_{X^*}Y^*)\a_1(\nabla_{U^*}V^*)+\a_2(\nabla_{X^*}Y^*)\a_2(\nabla_{U^*}V^*)\big)$$
$$ -\a_1(\nabla_{Y^*}U^*)\a_1(\nabla_{X^*}V^*)-\a_2(\nabla_{Y^*}U^*)\a_2(\nabla_{X^*}V^*)$$
$$+\a_1(\nabla_{X^*}U^*)\a_1(\nabla_{Y^*}V^*)+\a_2(\nabla_{X^*}U^*)\a_2(\nabla_{Y^*}V^*).$$

 Using these by straightforward computation we have the following.
 $$g((\nabla_{V^*}R)_{X^*Y^*}U^*,W^*)=g'((\nabla'_VR')_{XY}U,W)$$
$$+2g((\nabla_{V^*}\phi_1)X^*,Y^*)g(\phi_1U^*,W^*)
+2g(\phi_1X^*,Y^*)g((\nabla_{V^*}\phi_1)U^*,W^*)$$
$$+2g((\nabla_{V^*}\phi_2)X^*,Y^*)g(\phi_2U^*,W^*)
+2g(\phi_2X^*,Y^*)g((\nabla_{V^*}\phi_2)U^*,W^*)$$
$$-g((\nabla_{V^*}\phi_1)Y^*,U^*)g(\phi_1X^*,W^*)
-g(\phi_1Y^*,U^*)g((\nabla_{V^*}\phi_1)X^*,W^*)$$
$$-g((\nabla_{V^*}\phi_2)Y^*,U^*)g(\phi_2X^*,W^*)
-g(\phi_2Y^*,U^*)g((\nabla_{V^*}\phi_2)X^*,W^*)$$
$$+g((\nabla_{V^*}\phi_1)X^*,U^*)g(\phi_1Y^*,W^*)
+g(\phi_1X^*,U^*)g((\nabla_{V^*}\phi_1)Y^*,W^*)$$
$$+g((\nabla_{V^*}\phi_2)X^*,U^*)g(\phi_2Y^*,W^*)
+g(\phi_2X^*,U^*)g((\nabla_{V^*}\phi_2)Y^*,W^*).\eqno(4.1)$$

Linearizing the reflection condition 
$g((\nabla_{X^*}R)_{X^*Y^*}X^*,Y^*)=0$ in the Chen-Vanhecke Theorem
and using the second Bianchi identity 
(cf.\cite{C}, pp. 257-258) 
we see that the left hand side of the above vanishes.  The first term 
on the right vanishes since $M'$ is locally symmetric.  Moreover the
equation is tensorial, so the remainder of the equation holds for all
horizontal $X,Y,U,V,W$, i.e. we have
$$0=2g((\nabla_{V}\phi_1)X,Y)g(\phi_1U,W)
+2g(\phi_1X,Y)g((\nabla_{V}\phi_1)U,W)$$
$$+2g((\nabla_{V}\phi_2)X,Y)g(\phi_2U,W)
+2g(\phi_2X,Y)g((\nabla_{V}\phi_2)U,W)$$
$$-g((\nabla_{V}\phi_1)Y,U)g(\phi_1X,W)
-g(\phi_1Y,U)g((\nabla_{V}\phi_1)X,W)$$
$$-g((\nabla_{V}\phi_2)Y,U)g(\phi_2X,W)
-g(\phi_2Y,U)g((\nabla_{V}\phi_2)X,W)$$
$$+g((\nabla_{V}\phi_1)X,U)g(\phi_1Y,W)
+g(\phi_1X,U)g((\nabla_{V}\phi_1)Y,W)$$
$$+g((\nabla_{V}\phi_2)X,U)g(\phi_2Y,W)
+g(\phi_2X,U)g((\nabla_{V}\phi_2)Y,W).$$

We now make the following choices of unit vector fields: $W=X\in
T{\mathcal F}_1$;
$Y=\phi_2 X$; $U,V\in T{\mathcal F}_2$.
This gives
$$0=3g((\nabla_{V}\phi_2)U,X)=3g(-\phi_2\nabla_VU,X)
=3g(\nabla_VU,\phi_2 X).\eqno(4.2)$$

Next we will show that the integral submanifolds of ${\mathcal F}_2$ are
totally geodesic submanifolds in $M$. Again let $\sigma$ denote the second
fundamental form of an integral submanifold.  Since
$\nabla_{Z_1}Z_1=0$,
$\sigma(Z_1,Z_1)=0$.  For $U$ horizontal in $T{\mathcal F}_2$,
$\nabla_UZ_1=-\phi_1U$ which is again in $T{\mathcal F}_2$ and therefore
$\sigma(U,Z_1)=0$. Since
$\nabla_UZ_2=-\phi_2U=0$, $\sigma$ has no $Z_2$ component.
Finally from (4.2), since $\phi_2 X$ can be regarded
as any horizontal vector in $T{\mathcal F}_1$, $\sigma(U,V)=0$ for any
$U,V\in T{\mathcal F}_2$.  Similarly the integral submanifolds of 
${\mathcal F}_1$ are totally geodesic.

Since the integral submanifolds of ${\mathcal F}_2$ are
totally geodesic submanifolds, the structure tensors
$(\phi_1,Z_1,\alpha_1,g)$ induce a Sasakian structure on each integral
submanifold of ${\mathcal F}_2$ and similarly for the leaves of 
${\mathcal F}_1$.  Moreover for horizontal
$X,Y,U,V,W\in T{\mathcal F}_2$ equation (4.1) and the Sasakian condition readily yields
$$g((\nabla_{V}R)_{XY}U,W)=0$$
  and hence the integral submanifolds are
locally $\phi$-symmetric spaces \cite{Ta}.  Furthermore $M$ being  complete and simply connected, the integral submanifolds are globally $\phi$-symmetric spaces \cite{Ta}. Thus, since the leaves of the
characteristic foliations are totally geodesic  $M$ is the
product of globally $\phi$-symmetric spaces.  In the case of type $(m,0)$, the second factor is one dimensional.

Finally the tensor fields $d\alpha_i$ are also projectable since $\pounds_{Z_i}d\a_j=0$, $i,j=1,2$,
 giving $M'$ a symplectic pair structure.
\end{proof}

\bibliographystyle{amsplain}

\begin{thebibliography}{999}
 \bibitem{Bande1}
G.~Bande, \textsl{Formes de contact g{\'e}n{\'e}ralis{\'e},
couples de contact et couples contacto-symplectiques}, Th{\`e}se
de Doctorat, Universit{\'e} de Haute Alsace, Mulhouse, 2000.
 
 \bibitem{BH}
  G. Bande and A. Hadjar, {\it Contact pairs}, T\^ohoku Math. J. {\bf 57} (2005),247-260.
 
 \bibitem{BH1}
 G. Bande and A. Hadjar, {\it On normal contact pairs}, Internat. J. Math. {\bf 21} (2010), 737-754.

\bibitem{BH3}
G. Bande and A. Hadjar, {\it Contact pairs structures and associated metrics}, Differential Geometry - Proceedings of the 8th International Colloquium, World Sci. Publ. (2009), 266--275.

 \bibitem{BH2}
G. Bande, D. E. Blair and A. Hadjar, {\it On the curvature of metric contact pairs}, Preprint arXiv:0346639 [math.DG] 31 Oct 2011.

\bibitem{BK} G.~Bande and D.~Kotschick, \emph{Contact pairs and locally conformally symplectic structures}, Harmonic maps and differential geometry, 85--98, Contemp. Math., \textbf{542}, Amer. Math. Soc., Providence, RI, 2011..

 \bibitem{B}
D. E. Blair,
Riemannian Geometry of Contact and Symplectic Manifolds, Second Edition, Birkh\"auser, Boston, 2010.

\bibitem{BLY}
D. E. Blair, G. D. Ludden and K. Yano, {\it Geometry of complex manifolds similar to the Calabi-Eckmann manifolds}, J. Differential Geometry {\bf 9} (1974), 263-274.

 \bibitem{BV}
 D. E. Blair and L. Vanhecke,
{\it Symmetries and $\phi$-symmetric spaces}, T\^ohoku Math. J. {\bf 39} (1987), 373-383.

\bibitem{BluHeb}
R. A. Blumenthal and J. J. Hebda, {\it De Rham decomposition theorems for foliated manifolds},
Ann. Inst. Fourier, Grenoble {bf 33}, 183-198.

 \bibitem{BC}
E. Boeckx and J. T. Cho,
{\it Locally symmetric contact metric manifolds}, Monatsh. Math. {\bf 148} (2006), 269-281.

 \bibitem{C}
E. Cartan,
Geometry of Riemannian Spaces (English translation), Math. Sci. Press, Brookline, 1983.

\bibitem{CV}
B.-Y. Chen and L. Vanhecke,
 {\it Isometric, holomorphic and symplectic reflections}, Geometriae Dedicata {\bf 29} (1989), 259-277.

\bibitem{FIP}
M. Falcitelli, S. Ianus and A. Pastore,
Riemannian Submersions, World Scientific, Singapore, 2004.



\bibitem{O} 
M. Okumura,
{\it Some remarks on space with a certain structure}, T\^ohoku Math. J. {\bf 14} (1962), 
135-145. 

\bibitem{ON}
B. O'Neill, 
{\it The fundamental equations of a submersion}, Michigan Math. J. {\bf 13} (1966), 459-469.

\bibitem{Ta}
 T. Takahashi,
 {\it Sasakian $\phi$-symmetric spaces}, T\^ohoku Math. J. {\bf 29} (1977), 91-113.
 

 \bibitem{Tanno} S. Tanno,
 {\it Loocally symmetric K-contact Riemannian manifolds}, Proc. Japan Acad. {\bf 43} (1967), 581-583.

 \end{thebibliography}

\end{document}